\documentclass[11pt, notitlepage]{amsart}
\usepackage{a4wide}
\usepackage{amsthm,amsmath,amssymb,amscd,mathrsfs,enumerate,amsfonts}

\swapnumbers
\newtheorem{thm}{Theorem}[section]
\newtheorem{prop}[thm]{Proposition}
\newtheorem{lemma}[thm]{Lemma}
\newtheorem{cor}[thm]{Corollary}

\newtheorem{obser}[thm]{Observation}
\newtheorem*{INTRA}{Theorem Z}
\newtheorem*{INTR1}{Theorem 1}
\newtheorem*{INTR2}{Theorem 2}
\theoremstyle{definition}
\newtheorem{defn}[thm]{Definition}
\newtheorem*{ack}{Acknowledgment}

\newtheorem*{PRELnotation}{Notation}
\newtheorem*{PRELconvention}{Convention}
\theoremstyle{remark}
\newtheorem{remark}[thm]{Remark}

\def \dist {\operatorname{dist}}

\def\M{\mathcal{M}}

\def\U{\mathcal{U}}
\def\V{\mathcal{V}}
\def\I{\mathcal{I}}
 
\def\en{\mathbb N}

\def\er{\mathbb R}
\def\err{{\bf P}}
\def\ell{{\bf L}}
\def\uu{{\bf U}}

\def\eps{\varepsilon}  

\begin{document}
\author{Martin Rmoutil}
\title{Products of non-$\sigma$-lower porous sets}
\thanks{The work was supported by the Grant SVV-2011-263316.}
\email{mar@rmail.cz}
\address{Charles University, Faculty of Mathematics and Physics, Department of Mathematical\linebreak Analysis, Sokolovsk\'a 83, 186 75 Prague 8, Czech Republic}
\begin{abstract}
In the present article we provide an example of two closed non-$\sigma$-lower porous sets $A, B \subseteq \er$ such that the product $A\times B$ is lower porous.
On the other hand, we prove the following: Let $X$ and $Y$ be topologically complete metric spaces, let $A\subseteq X$ be a non-$\sigma$-lower porous Suslin set and let $B\subseteq Y$ be a non-$\sigma$-porous Suslin set. Then the product $A\times B$ is non-$\sigma$-lower porous. We also provide a brief summary of some basic properties of lower porosity, including a simple characterization of Suslin non-$\sigma$-lower porous sets in topologically complete metric spaces. 
\end{abstract}
\maketitle
\section{Introduction}
In the present article we deal with Cartesian products of $\sigma$-lower porous sets. The work is motivated by a paper of L.~Zaj\'i\v{c}ek \cite{1} where the following theorem is proved:
\begin{INTRA}[{[\ref{1}, Theorem 1]}]
Let $(X,\rho)$ and $(Y,\sigma)$ be topologically complete metric spaces and let $A\subseteq X$ and $B\subseteq Y$ be non-$\sigma$-porous $G_\delta$-sets. Then the Cartesian product $A\times B$ is non-$\sigma$-porous in the space $(X\times Y,\rho_m)$ where $\rho_m$ is the maximum metric.
\end{INTRA}
It is a natural question to ask, whether an analogous statment holds for lower porosity (i.e. the notion of porosity defined by limes inferior rather than limes superior). In the third section of this article we show that the answer is generally negative by giving a counterexample.

\begin{INTR1}
There exist closed non-$\sigma$-lower porous sets $A\subseteq \er$ and $B\subseteq \er$ such that the Cartesian product $A\times B$ is lower porous in $\er^2$.
\end{INTR1}

However, if we strengthen the assumptions of the original conjecture, we obtain the following theorem. These two theorems together give us a fairly complete answer to our question.

\begin{INTR2}
Let $(X,\rho)$ and $(Y,\sigma)$ be topologically complete metric spaces. Assume that $A\subseteq X$ and $B\subseteq Y$ are Suslin sets in their respective spaces. If $A$ is non-$\sigma$-lower porous in $X$ and $B$ is non-$\sigma$-porous in $Y$ then the Cartesian product $A\times B$ is non-$\sigma$-lower porous in $X\times Y$ (with the maximum metric).
\end{INTR2}
It is easy to see that both aforementioned notions of $\sigma$-porosity are invariant with respect to bilipschitz homeomorphisms. Therefore, in all the previous theorems we can equip the product spaces with any metric which is ``bilipschitz equivalent" to the maximum metric and the resulting statement will be true.

It is also fitting to give an explanation as to why in Theorem 2 we only require the sets $A$ and $B$ to be Suslin while in Theorem Z these are assumed to be of the type $G_\delta$. The reason is that we use two inscribing theorems (see \ref{inscrupper} and \ref{inscrlower}) which, at the time Theorem Z was proved, had not yet been discovered. Of course, this means Theorem Z can be generalized to Suslin sets.
\section{Some facts about $\sigma$-lower porosity and abstract porosity}
The main aim of this section is to provide the reader with a self-contained collection of some basic facts about $\sigma$-lower porous sets (with some references to related articles). It might be of some independent interest, but we shall use these facts to prove our main results.
\begin{PRELnotation}
In the whole paper we shall denote by $B(x,r)$ the open ball with centre $x$ and radius $r$, by $\overline{A}$ the closure of the set $A$, and by $\partial A$ the boundary of $A$. As usual, for a set $X$ the symbol $2^X$ denotes the power set of $X$. \end{PRELnotation}
\begin{PRELconvention}
Unless stated otherwise, we shall consider all product spaces equipped with the maximum metric (i.e. for $x_1,x_2\in (X,\rho)$ and $y_1,y_2\in (Y,\sigma)$, $\rho_m([x_1;y_1],[x_2;y_2]) =$ \linebreak $\max\left\{ \rho(x_1,x_2), \sigma(y_1,y_2) \right\}$).
\end{PRELconvention}
The following standard definitions of $\sigma$-porosity originate in a work of A. Denjoy from 1920; however, a systematic investigation of these sets (and the usage of the current nomenclature) has begun in 1967 with an article of E.~P.~Dolzhenko.  For extensive information about $\sigma$-porous sets from various viewpoints we refer the reader to L.~Zaj\'i\v{c}ek's survey articles 
\linebreak \cite{7} and \cite{2}. 

\begin{defn}\label{def1}
Let $(X,\rho)$ be a metric space, $M\subseteq X$, $x\in X$ and $R>0$. We define 
\[ \gamma (x,R,M)=\sup\left\{ r>0: \mbox{for some } z\in X, \; B(z,r)\subseteq B(x,R)\setminus M \right\}, \]
\[ \overline{p}(M,x)=\limsup_{R\to 0_+}\frac{2\cdot\gamma(x,R,M)}{R}, \qquad \underline{p}(M,x)=\liminf_{R\to 0_+}\frac{2\cdot\gamma(x,R,M)}{R}. \]
A set $M\subseteq X$ is \emph{(upper) porous at x} if $\overline{p}(M,x)>0$ and \emph{lower porous at x} if $\underline{p}(M,x)>0$.

Now assume $\err$ is a relation between points and subsets of $X$ (i.e. $\err\subseteq X\times 2^X$). The symbol $\err(x,A)$ where $x\in X$ and $A\subseteq X$ means that $[x;A]\in \err$. We say that $\err$ is an \emph{abstract porosity on $X$} if the following conditions are satisfied (for all $A\subseteq X$, $B\subseteq X$ and $x\in X$):

\begin{itemize}
\item[\bf(A1)] If $A\subseteq B\subseteq X$, $x\in X$ and $\err(x,B)$, then $\err(x,A)$.
\item[\bf(A2)] $\err(x,A)$ if and only if there is an $r>0$ such that $\err(x,A\cap B(x,r))$.
\item[\bf(A3)] $\err(x,A)$ if and only if $\err(x,\overline{A})$.
\end{itemize}

Note that the relations which correspond (in the sense of the first point of the following list) to the notions of porosity and lower porosity are clearly abstract porosities. Let $\err$ be an abstract porosity on $X$. We say that $A\subseteq X$ is

\begin{itemize}
\item \emph{$\err$-porous at $x\in X$} if $\err(x,A)$,
\item \emph{$\err$-porous (in X)} if $A$ is $\err$-porous at each of its points,
\item \emph{$\sigma$-$\err$-porous (in X)} if $A$ is a countable union of $\err$-porous sets,
\item \emph{$\sigma$-$\err$-porous at $x\in X$} if there is an $r>0$ such that $A\cap B(x,r)$ is $\sigma$-$\err$-porous.
\end{itemize}

In case $\err$ corresponds to lower porosity we say \emph{$A$ is lower porous}, \emph{$\sigma$-lower porous} or \emph{$\sigma$-lower porous at $x$}. If $\err$ corresponds to ordinary (upper) porosity, we simply omit the symbol $\err$ and write \emph{$A$ is porous} etc. (however, in some cases we tend to add ``upper'' to avoid confusion).
\end{defn}

\begin{remark}
If $(X,\rho)$ is a metric space and $\err$ is an abstract porosity on $X$, it is well known that the system $\I$ of all $\sigma$-$\err$-porous sets in $X$ satisfies the following conditions:
\begin{enumerate}[(i)]
\item If $A\subseteq B$ and $B\in\I$ then $A\in \I$.
\item If $A_n\in \I$ for all $n\in\en$ then $\bigcup_{n=1}^\infty A_n \in\I$.
\end{enumerate}

\end{remark}

Our first step will be to develop a method to recognize non-$\sigma$-lower porous sets (Proposition \ref{propBaire}). The following two propositions are well known (see the survey article \cite{2}), but we shall provide the proofs for the sake of completeness.
\begin{prop}\label{lemmaBaire}
Let $(X,\rho)$ be a metric space and let $A\subseteq X$ be $\sigma$-lower porous. Then $A$ can be covered by a countable system of closed lower porous sets.
\end{prop}
\begin{proof}
Without loss of generality we can assume the set $A$ is lower porous. From the definition of lower porosity is clear that for any $x\in A$ we can choose a positive number $h_0=h_0(x)$ such that for all $h\in (0,h_0(x))$:
\[ \frac{2\cdot\gamma(x,h,A)}{h}>\frac{\underline{p}(x,A)}{2}.\]
Thus we have chosen a function $h_0:A\longrightarrow (0,\infty)$. Set
\[ A_n:= \Big\{ x\in A: h_0(x)>\frac{1}{n} \;\mbox{ and }\; \underline{p}(x,A)>\frac{1}{n} \Big\};\]
then, clearly, $A=\bigcup_{n=1}^\infty A_n$. We shall now prove that for each $n\in\en$ the set $\overline{A_n}$ is lower porous. And since it is obvious that for any $x\in X$, $R>0$ and $M\subseteq X$ the equality $\gamma(x,R,M)=\gamma(x,R,\overline{M})$ is true, we only need to show that the set $A_n$ is lower porous at each point $x\in \overline{A_n}\setminus A_n$.

To that end, choose a natural number $n$ and a point $x\in \overline{A_n}\setminus A_n$. Now, for an arbitrary $h\in \left(0,\frac{1}{n}\right)$ there is a point $y\in B\left(x,\frac{h}{2}\right)\cap A_n$ and from the definition of $A_n $ it follows that there is a point $z\in B\left(y,\frac{h}{2}\right)$ such that $B\left(z,\frac{h}{8n}\right)\subseteq B\left(y,\frac{h}{2}\right)\setminus A_n$. Thus $\gamma(x,h,A_n)\geq \frac{h}{8n}$ and
\[ \liminf_{h\rightarrow 0_+}\frac{2\cdot\gamma(x,h,A_n)}{h}\geq \frac{1}{4n}>0.\]
\end{proof}
\begin{prop}\label{propBaire}
Let $(X,\rho)$ be a metric space and let $F\subseteq X$ be a topologically complete subspace. Let there exist a set $D\subseteq F$ dense in $F$ such that $F$ is lower porous (in $X$) at no point $x\in D$. Then $F$ is not $\sigma$-lower porous in $X$.
\end{prop}
\begin{proof}
Assume to the contrary that $F$ is $\sigma$-lower porous. Proposition \ref{lemmaBaire} gives us closed lower porous sets $F_n$ ($n\in \en$) such that $F\subseteq \bigcup_{n=1}^\infty F_n$. Hence $F=\bigcup_{n=1}^\infty (F_n\cap F)$ and the set $F_n\cap F$ is closed in $F$ for each natural $n$. Using the Baire theorem in the topologically complete space $F$ we obtain an open set $G\subseteq X$ such that $\emptyset\neq G\cap F \subseteq F_{n_0}\cap F$ for some natural number $n_0$. Thus $G\cap F$ (being a subset of $F_{n_0}$) is lower porous in $X$ and it follows that $F$ is lower porous at every point $x\in G\cap F$ (for $G$ is an open set). But the set $A$ is dense in $F$ so there exists a point $x\in A\cap G\cap F$ which is a contradiction with the assumption that $F$ is lower porous at no point of $A$.
\end{proof}

Now we formulate two rather deep inscribing theorems which will be used on various occasions throughout the paper. Their purpose is to obtain some of our statements about non-$\sigma$-porous and non-$\sigma$-lower porous sets for all Suslin sets instead of closed (or G$_\delta$) sets only.

\begin{thm}[{[\ref{6}, Theorem 3.1]}]\label{inscrupper}
Let $(X,\rho)$ be a topologically complete metric space and let $S\subseteq X$ be a non-$\sigma$-porous Suslin set. Then there exists a closed non-$\sigma$-porous set $F\subseteq S$.
\end{thm}

\begin{thm}[{[\ref{5}, Corollary 3.4]}]\label{inscrlower}
Let $(X,\rho)$ be a topologically complete metric space and let $S\subseteq X$ be a non-$\sigma$-lower porous Suslin set. Then there exists a closed non-$\sigma$-lower porous set $F\subseteq S$.
\end{thm}
We continue by recalling several basic definitions (cf. e.g. \cite{3} and \cite{4}) which we need in the following.
\begin{defn}
Let $(X,\rho)$ be a metric space and let $\err$ be an abstract porosity on $X$. If $A\subseteq X$ then by $K_\err (A)$ we denote the set of all $x\in A$ such that $A$ is not $\sigma$-$\err$-porous at $x$.

A system $\M\subseteq 2^X$ is called
\begin{itemize}
\item \emph{locally finite} if for each $x\in X$ there is an $r>0$ such that the ball $B(x,r)$ intersects at most finitely many elements of $\M$,
\item \emph{discrete} if for each $x\in X$ there is an $r>0$ such that the ball $B(x,r)$ intersects at most one element of $\M$,
\item \emph{$\sigma$-discrete} if it is a countable union of discrete systems.
\end{itemize}

We say $\M$ is a \emph{cover} of $X$ if $\bigcup \M = X$. Let $\U$ and $\V$ be two covers of $X$. Then $\V$ is a \emph{refinement} of $\U$ if for each $B\in \V$ there is a set $A\in \U$ such that $B\subseteq A$.
\end{defn}
An elementary proof of the following Proposition \ref{propStone} can be found as the proof of Lemma 3 in the article \cite{3}; we give an alternative proof which is more transparent, but is not elementary since it uses the famous theorem of A.~H.~Stone about the paracompactness of metric spaces ([\ref{4}, Theorem 4.4.1]). We will use the following easy lemma.
\begin{lemma} \label{lemmadiscr}
Let $(X,\rho)$ be a metric space and let $\err$ be an abstract porosity on $X$. Then:
\begin{enumerate}[(i)]
\item If $\M$ is a discrete system of $\err$-porous sets, then $\bigcup \M$ is $\err$-porous.
\item If $\M$ is a $\sigma$-discrete system of $\sigma$-$\err$-porous sets, then $\bigcup \M$ is $\sigma$-$\err$-porous.
\end{enumerate}
\end{lemma}
\begin{proof}
First, we shall prove assertion (i). Let $\M$ be a discrete system of $\err$-porous sets and let $x\in\bigcup\M$ be an arbitrary point; we shall prove that $\bigcup \M$ is $\err$-porous at $x$. Since the system $\M$ is discrete, there is an $r>0$ and $M\in\M$ such that
\begin{equation}\label{4equ1}
\left(\bigcup\M \right)\cap B(x,r) = M\cap B(x,r). 
\end{equation}
The set $M$ is $\err$-porous and from {\bf (A1)} (see \ref{def1}) we have that so is $M\cap B(x,r)$. It follows from \eqref{4equ1} and {\bf (A2)} that also the sum $\bigcup \M$ is $\err$-porous at $x$.

To prove the second assertion, assume (clearly without loss of generality) $\M$ is a discrete system of $\sigma$-$\err$-porous sets. Each $M\in\M$ can be written in the form $M=\bigcup_{n=1}^\infty A_n^M$ where the set $A_n^M$ is $\err$-porous for any $n\in \en$. It is obvious that for each $n\in\en$ the system $\left\{ A_n^M: M\in\M \right\}$ is discrete. Thus, using the first part of this lemma, we obtain the $\sigma$-$\err$-porosity of
\[ \bigcup \M=\bigcup_{n=1}^\infty \: \bigcup_{M\in\M}A_n^M. \qedhere \]
\end{proof}

\begin{prop}\label{propStone}
Let $(X,\rho)$ be a metric space and let $\err$ be an abstract porosity on $X$. Assume the set $A\subseteq X$ is $\sigma$-$\err$-porous at each of its points. Then $A$ is $\sigma$-$\err$-porous.
\end{prop}
\begin{proof}
Set $A_n:=\left\{ x\in A: B \left(x, \frac{1}{n} \right) \cap A \mbox { is } \sigma\mbox{-}\err\mbox{-porous} \right\}$; by the assumption, $A=\bigcup_{n=1}^\infty A_n$. Let us fix an arbitrary $k\in \en$ and prove that $A_k$ is $\sigma$-$\err$-porous. 

To that end, we define the open cover $\U$ of $X$ as $\U:=\left\{ B\left( x, \frac{1}{2k}\right): x\in X \right \}$; it is easy to see that for each $B\in\U$ the set $B\cap A_k$ is $\sigma$-$\err$-porous. Using the Stone Theorem we obtain a $\sigma$-discrete refinement $\V$ of $\U$. Since $\V$ is a refinement of $\U$, we have that for each $G\in \V$ the set $G\cap A_k$ is $\sigma$-$\err$-porous and it follows from Lemma \ref{lemmadiscr} that $A_k=\bigcup\left\{G\cap A_k : G\in\V \right\}$ is $\sigma$-$\err$-porous. This concludes the proof.
\end{proof}

An immediate consequence of this result is the following.

\begin{cor}\label{kercor}
Let $(X,\rho)$ be a metric space and let $\err$ be an abstract porosity on $X$. Assume the set $A\subseteq X$ is not $\sigma$-$\err$-porous. Then:
\begin{enumerate}[(i)]
\item $K_\err (A)$ is nonempty and closed in $A$.
\item The set $A\setminus K_\err (A)$ is $\sigma$-$\err$-porous.
\item $K_\err( K_\err (A)) = K_\err (A)$ (i.e. $K_\err (A)$ is $\sigma$-$\err$-porous at none of its points).
\item The set of all points at which $K_\err(A)$ is not $\err$-porous is dense in $K_\err (A)$.
\end{enumerate}
\end{cor}

The proposition that follows now, provides a simple characterization of non-$\sigma$-lower porous Suslin sets. It can be regarded as an analogue for lower porosity to a partial converse of the Foran lemma which was proved by L.~Zaj\'i\v{c}ek (see [\ref{1}, Corollary 1]); the mentioned result works for upper porosity and $G_\delta$ sets (but can, of course, be generalized to Suslin sets via Theorem \ref{inscrupper}). 

\begin{prop}\label{charLP}
Let $(X,\rho)$ be a topologically complete metric space and let $A\subseteq X$ be a Suslin set. Then the following statements are equivalent:
\begin{enumerate}[(i)]
\item A is not $\sigma$-lower porous.
\item There exists a closed set $F\subseteq A$ and a set $D\subseteq F$ dense in $F$ such that $F$ is lower porous at no point of $D$.
\end{enumerate}
\end{prop}

\begin{proof}
To prove the implication (i)$\Rightarrow$(ii) assume A is a non-$\sigma$-lower porous Suslin set; using Theorem \ref{inscrlower} we can assume without loss of generality that $A$ is closed. Now it suffices to take $F:=K(A)$, as all the desired properties of $F$ follow from Corollary \ref{kercor}.

To prove (ii)$\Rightarrow$(i) suppose that (ii) holds. Then Proposition \ref{propBaire} gives that $F$ is non-$\sigma$-lower porous, and thus so is $A$.
\end{proof}
\begin{remark}
\mbox{}
\begin{enumerate}[(a)]
\item It could be interesting to note a connection of Proposition \ref{charLP} to the article \cite{8} (especially Section 5) where the notion of $\err$-reducible sets is defined and studied. If $\err$ is an abstract porosity on a metric space $X$, we say that $A\subseteq X$ is $\err$-reducible if each nonempty closed set $F\subseteq A$ contains a $\err$-porous subset with nonempty relative interior in $F$. Now the statement of Proposition \ref{charLP} can be reformulated as follows:

If $X$ is topologically complete and $\ell$ is the relation corresponding to the notion of lower porosity on $X$, then a Suslin set $A\subseteq X$ is $\sigma$-lower porous if and only if it is $\ell$-reducible. 

\item Now let us briefly turn our attention to the general case. As Corollary \ref{kercor} (iv) holds for any abstract porosity $\err$, the following is true:

Let $\err$ be any abstract porosity on a metric space $X$ and let $A\subseteq X$ be closed. Then: 
\[ A \text{ is non-$\sigma$\text-$\err$-porous }\Longrightarrow A \text{ is not $\err$-reducible.}\] 
If $X$ is topologically complete and $\err$ corresponds to upper porosity on $X$, it suffices to assume the set $A$ to be Suslin (due to Theorem \ref{inscrupper}).

However, if $\err$ is such that an equivalent of Proposition \ref{propBaire} for $\err$ does not hold (e.g., the upper porosity), then the other implication in the previous statement does not necessarily hold (see Example 3.1 or Corollary 5.3 with Proposition 5.1 of [\ref{8}]). That is the reason why a more elaborate method of recognizing non-$\sigma$-upper porous sets had to be developed in order to prove Theorem Z from the introduction (the method of the Foran Lemma and its partial converse).

\end{enumerate}
\end{remark}

\section{Counterexample}
\begin{defn}
Denote $D_0:=\emptyset$ and for each $n\in\en$ we define the open set $D_n\subseteq(0,1)$ as
\[ D_n:=\bigcup_{i=0}^{3^{n-1}-1}\Big(\frac{1+3i}{3^n},\frac{2+3i}{3^n} \Big ). \]
Furthermore, for each $n\in {\mathbb N}\cup\{ 0 \}$ we define
\[ M_n:=\partial D_n, \qquad A_n:=[0,1]\setminus D_n. \]
Finally, if $I\subseteq {\mathbb N}$ is nonempty, we define
\[ D_I:=\bigcup_{n\in I}D_n,\qquad M_I:=\bigcup_{n\in I}M_n,\qquad A_I:=[0,1]\setminus D_I.\]
\end{defn}
\begin{defn}
Let $(X,\rho)$ be a metric space and let $\eps>0$. We say that $M\subseteq X$ is an \emph{$\eps$-net in $X$}, if for each point $x\in X$ there exists some $y\in M$ such that $\rho(x,y)\leq\eps$.
\end{defn}
The following facts are easy to see.
\begin{obser}\label{OBS1}
\mbox{}
\begin{enumerate}[(i)]
\item For each $n\in \en$ the set $M_n$ is a $3^{-n}$-net in the interval $[0,1]$.
\item If $I\subseteq \en$ is infinite, then
 \begin{itemize}
 \item $\overline{M_I}=[0,1]$,
 \item $A_I$ is porous.
 \end{itemize}
\item Whenever $m,n\in \en$, $m\neq n$, then we have $M_m\cap M_n=\emptyset$.
\item $M_n \cap D_m\neq \emptyset$ if and only if $m<n$.
\item $A_\en$ is the Cantor ternary set.
\end{enumerate}
\end{obser}
\begin{lemma}\label{lemmadense}
Let $I\subseteq \en$ be infinite and let $\emptyset\neq J\subseteq \en$. Then $M_I\cap A_J$ is dense in $A_J$.
\end{lemma}
\begin{proof}
Choose an arbitrary $y\in A_J$ and $\eps>0$. Now find an $n_0\in I$ such that $2\cdot3^{-n_0}<\eps$ and denote $K:=J\cap (0,n_0)$. On account of \ref{OBS1} (iv) it is true that $M_{n_0}\cap A_K=M_{n_0}\cap A_J$. Setting $n_1:=\max (K\cup\{ 0 \})$ we have $n_1<n_0$ and it is obvious that the components of $A_K$ are closed intervals whose length is at least $3^{-n_1}$. The set $M_{n_0}$ is a $3^{-n_0}$-net in $[0,1]$ (\ref{OBS1} (i)) and $3^{-n_1}>2\cdot 3^{-n_0}$; from these two facts now easily follows that $M_{n_0}$ is a $(2\cdot 3^{-n_0})$-net in $A_K$. This implies the existence of a point $z\in M_{n_0}\cap A_K=M_{n_0}\cap A_J\subseteq M_I\cap A_J$ such that $|z-y| \leq 2\cdot 3^{-n_0} <\eps$, which concludes the proof.
\end{proof}
\begin{defn}\label{defdelta}
Let $(X,\rho)$ be a metric space, let $A\subseteq X$ and let $x\in X$. We define the function $\delta_{A,x}:(0,\infty) \longrightarrow [0,\infty)$ as
$$\delta_{A,x}(h):=\frac{2\cdot\gamma(x,h,A)}{h}.$$
\end{defn}
\begin{lemma}\label{lemmainter}
Assume $I\subseteq \en$ is nonempty and let $x\in A_I$ and $n\in I$. Then for each $h\in \left[ \frac{4}{3^{n+1}}, \frac{4}{3^n} \right]$ we have that  $\delta_{A_I,x}(h)\geq \frac{1}{4}$.
\end{lemma}
\begin{proof}
Let $I\subseteq \en$, $x\in A_I$ and $n\in I$ be given. Since $n\in I$, we have that $A_I\subseteq [0,1] \setminus D_n$ and thus $x\in [0,1] \setminus D_n$. From \ref{OBS1} (i) we know that the set $M_n= \partial D_n$ is a $3^{-n}$-net in the interval $[0,1]$ which implies that $\dist(x,D_n)\leq 3^{-n}$. From this and from the fact that $D_n$ consists of pairwise disjoint open intervals of length $3^{-n}$, it follows that for all $h\in \left[\frac{1}{3^n},\frac{2}{3^n}\right]$ holds the inequality
\[ 2\cdot\gamma(x,h,[0,1]\setminus D_n) \geq h-\frac{1}{3^n}. \]
What is more, for any $h>\frac{2}{3^n}$
\[ 2\cdot\gamma(x,h,[0,1]\setminus D_n) \geq \frac{1}{3^n}. \]
Consequently
\[ \delta_{A_I,x}(h)\geq \frac{2\cdot \gamma(x,h,[0,1] \setminus D_n)}{h} \geq \]
\begin{equation*}
\geq \begin{cases} \frac{1}{h}\left(h-\frac{1}{3^n}\right)\geq 1-\frac{3^{n+1}}{4}\cdot \frac{1}{3^n}=\frac{1}{4} & \text{ for } h\in\left[\frac{4}{3^{n+1}},\frac{2}{3^n}\right],\\
						
										 \frac{1}{h}\cdot\frac{1}{3^n}\geq \frac{3^n}{4}\cdot \frac{1}{3^n}=\frac{1}{4} & \text{ for } h\in\left[\frac{2}{3^n},\frac{4}{3^n}\right].
\end{cases}
\qedhere
\end{equation*}
\end{proof}
\begin{prop}\label{cormax}
Let $(X,\rho)$ and $(Y,\sigma)$ be metric spaces and let us have sets $A\subseteq X$ and $B\subseteq Y$. Finally, let there be given points $x\in X$ and $y\in Y$. Then: 
\begin{enumerate}[(i)]
\item $\gamma([x;y],h,A\times B)=\max \{ \gamma(x,h,A)\, , \gamma(y,h,B) \}$ for any $h>0$.
\item $\delta_{A\times B,[x;y]}=\max \{ \delta_{A,x} \, , \delta_{B,y} \}$.
\end{enumerate}
\end{prop}
\begin{proof}
We shall prove assertion (i). Without loss of generality we may assume that $\alpha:=\max \{ \gamma(x,h,A)\, , \gamma(y,h,B) \}=\gamma(x,h,A)>0$. Choose arbitrary $h>0$ and $\eps\in(0,\alpha)$. By the definition of $\gamma(x,h,A)$, there exists a point $x_1\in X$ such that $B(x_1,\alpha - \eps) \subseteq B(x,h) \setminus A$. Thus,
\[ B([x_1;y],\alpha-\varepsilon)\subseteq B([x;y],h)\setminus A\times B \]
and this means that
\[ \gamma([x;y],h,A\times B)\geq \alpha-\varepsilon=\max \{ \gamma(x,h,A)\, , \gamma(y,h,B) \}-\varepsilon.  \]
To prove the opposite inequality we take arbitrary $h>0$ and $\eps>0$ again. Setting $\beta:=\gamma([x;y],h,A\times B)$, we can assume that $\eps<\beta$. Now find a point $[x_1;y_1]\in X\times Y$ such that
\[ G:=B([x_1;y_1],\beta-\varepsilon)\subseteq B([x;y],h)\setminus A\times B. \]
Taking into account that $G=B(x_1,\beta-\varepsilon)\times B(y_1,\beta-\varepsilon)$ (for we consider the space $X\times Y$ with the maximum metric), this yields that
\[ B(x_1,\beta-\varepsilon)\subseteq B(x,h)\setminus A \qquad \text{or}\qquad B(y_1,\beta-\varepsilon)\subseteq B(y,h)\setminus B.\]
This implies the following inequality which concludes the proof of (i):
\[ \max \{ \gamma(x,h,A)\, , \gamma(y,h,B) \} \geq \beta - \varepsilon = \gamma([x;y],h,A\times B) - \varepsilon. \]
The second assertion follows immediately from (i).
\end{proof}
\begin{cor}\label{cordelmax}
Under the assumptions of Proposition \ref{cormax} we have that if $A$ is not lower porous at $x$ and $B$ is not porous at $y$, then $A\times B$ is not lower porous at $[x;y]$.
\end{cor}
\begin{proof}
Let $x$ and $y$ be as above. Then
\[ \liminf_{h\to 0_+}\delta_{A,x}(h)=0 \qquad \text{and} \qquad \limsup_{h\to 0_+}\delta_{B,y}(h)=0. \]
From Proposition \ref{cormax} we know that $\delta_{A\times B,[x;y]}=\max \{ \delta_{A,x} \, , \delta_{B,y} \}$,
and so it is easy to see that $\liminf_{h\to 0_+}\delta_{A\times B,[x;y]}=0$, i.e., $A\times B$ is not lower porous at $[x;y]$.
\end{proof}
\begin{remark}
Let $(X,\rho)$ be a metric space. If the set $A\subseteq X$ is not porous then neither is $A^2=A\times A$ porous in $X^2$. The same statement is true for lower porosity or, in general, for any notion of porosity which is determined solely by the function $\delta_{A,x}(h)$.

Indeed, if we assume that the set $A$ is not porous at a certain point $x\in A$, then, since $\delta_{A,x}=\delta_{A^2,[x;x]}$, it is clear that $A^2$ is not porous at $[x;x]$. Clearly, the same argument works for many other notions of porosity -- including, for example, lower porosity.
\end{remark}
We shall now prove the main result of this section which implies Theorem 1 from the Introduction.
\begin{thm}
Let the set $I\subseteq \en$ be defined by the formula
\[ I:=\bigcup_{i=1}^\infty\,\left[\, i^2,i^2+i\,\right)\cap\en\]
and let $J=\en\setminus I$. Then none of the closed sets $A_I$ and $A_J$ is $\sigma$-lower porous while the product $A_I \times A_J$ is lower porous.
\end{thm}
\begin{proof}
First, we shall prove that the set $A_J$ is not $\sigma$-lower porous; of course, the proof for $A_I$ would be analogous. Being a closed subspace of $\er$, $A_J$ is a topologically complete space. Hence, according to Proposition \ref{propBaire} it suffices to find a dense subset of $A_J$ at whose points the set $A_J$ is not lower porous. We claim that $M_I\cap A_J$ is such a set. Indeed, by Lemma \ref{lemmadense}, $M_I\cap A_J$ is dense in $A_J$; it only remains to be shown that $A_J$ is lower porous at no point of $M_I\cap A_J$.

To prove that, choose an arbitrary point $x\in M_I \cap A_J$ and let $n_0\in I$ be the unique natural number such that $x\in M_{n_0}$ (the uniqueness of $n_0$ is clear from \ref{OBS1} (iii)). Now $x$ can be written in the form $\frac{k}{3^{n_0}}$, where $k\in\en$ is not divisible by $3$. It follows that for each natural $j>n_0$
\begin{equation}\label{equ1}
\dist(x,D_j)=\frac{1}{3^j}.
\end{equation}
Moreover, since $x\in M_I\cap A_J$, for each natural $j<n_0$ we have
\begin{equation}\label{equ2}
\dist(x,D_j)\geq\frac{1}{3^{n_0}}.
\end{equation}
Now fix a natural number $i_0$ such that $i_0^2>n_0$ and choose an arbitrary $i>i_0$. The inequalities \eqref{equ1} and \eqref{equ2} imply that
\begin{equation}\label{equ3}
\dist\Big( x, \bigcup\,\{ D_n : n\in J, n\leq i^2-1 \} \Big) = \frac{1}{3^{i^2-1}}.
\end{equation}
From the definition of $J$ we see that $\left\{ i^2, i^2+1,\dots, i^2+i-1\right\}\cap J=\emptyset$. This fact, together with \eqref{equ3}, implies that the longest interval contained in 
\[ \Bigl(x- \frac{1}{3^{i^2-1}},x+\frac{1}{3^{i^2-1}}\Bigr) \]
and disjoint with $A_J$ is a component of $D_{i^2+i}$ (as $i^2+i\in J$), and therefore its length is $3^{-(i^2+i)}$. That is,
\[ \delta_{A_J,x}\Big(\frac{1}{3^{i^2-1}}\Big)=3^{i^2-1}\cdot\frac{1}{3^{i^2+i}}=\frac{1}{3^{i+1}};\]
it follows that $\liminf_{h\to 0_+}\delta_{A_J,x}(h)=0$ which means that $A_J$ is not lower porous at $x$.

To prove that the product $A_I \times A_J$ is lower porous, choose an arbitrary point $[x;y]\in A_I\times A_J$. By Lemma \ref{lemmainter} we have
\begin{align*}
										 & \delta_{A_I,x}(h)\geq \frac{1}{4},\quad\text{whenever}\quad h\in\bigcup_{n\in I}\Big[\frac{4}{3^{n+1}},\frac{4}{3^n}\Big]=:F_I, \\
\text{and also}\quad & \delta_{A_J,y}(h)\geq \frac{1}{4},\quad\text{whenever}\quad h\in\bigcup_{n\in J}\Big[\frac{4}{3^{n+1}},\frac{4}{3^n}\Big]=:F_J.
\end{align*}
But $I\cup J=\en$, so $F_I\cup F_J = \left(0, \frac{4}{3}\right]$, and it immediately follows from Proposition \ref{cormax} that \linebreak $\liminf_{h\to 0_+} \delta_{A_I \times A_J,[x;y]}(h) \geq \frac{1}{4}$, concluding the proof.
\end{proof}

\section{One positive result}

\begin{thm}
Let $(X,\rho)$ and $(Y,\sigma)$ be topologically complete metric spaces. Assume the Suslin set $A\subseteq X$ is not $\sigma$-lower porous and the Suslin set $B\subseteq Y$ is not $\sigma$-porous. Then the Cartesian product $A\times B$ is not $\sigma$-lower porous in the space $X\times Y$ (with the maximum metric).
\end{thm}
\begin{proof}
Let $\ell\subseteq X\times 2^X$ be the relation corresponding to the notion of lower porosity on $X$ (i.e. $\ell(x,C)$ if and only if $C$ is lower porous at $x$) and let $\uu\subseteq Y\times 2^Y$ be the relation corresponding to upper porosity on $Y$. Since both these relations are abstract porosities, from Corollary \ref{kercor} we know that $K_\ell(A)\neq \emptyset$ and $K_\uu(B)\neq \emptyset$; without loss of generality we shall now assume that $A=K_\ell(A)$ and $B=K_\uu(B)$ and using Theorem \ref{inscrupper} and Theorem \ref{inscrlower} we may also assume that the sets $A$ and $B$ are closed in their spaces.	

Denote by $A_1$ the set of all points of $A$ at which $A$ is not lower porous and by $B_1$ the set of all points of $B$ at which $B$ is not porous. From \ref{kercor} we know that $A_1$ is dense in $A$ and $B_1$ is dense in $B$; thus $A_1 \times B_1$ is dense in $A\times B$. By Proposition \ref{propBaire}, it suffices to prove that $A\times B$ is lower porous at no point of $A_1 \times B_1$. However, this is true due to Corollary \ref{cordelmax}, hence the proof is complete.
\end{proof}

\begin{ack}
I would like to thank Prof. Lud\v{e}k Zaj\'i\v{c}ek for suggesting the topic of this article and many useful remarks.
\end{ack}

\end{document}